\documentclass[12pt]{article}

\usepackage{amsmath}
\usepackage{amsfonts}
\usepackage{amssymb}
\usepackage{amsthm}
\usepackage{graphicx}
\usepackage{color}
\usepackage{hyperref}
\usepackage{enumitem}
\usepackage{float}

\usepackage[top=1.2in,bottom=1.2in,right=1in,left=1in]{geometry}

\theoremstyle{plain}

\newtheorem{theorem}{Theorem}[section]

\newtheorem{definition}[theorem]{Definition}

\newtheorem{lemma}[theorem]{Lemma}
\newtheorem{proposition}[theorem]{Proposition}
\newtheorem{remark}[theorem]{Remark}

\numberwithin{equation}{section}

\newcommand{\tr}{\operatorname{tr}}
\newcommand{\rank}{\operatorname{rank}}
\newcommand{\diag}{\operatorname{diag}}

\newcommand{\R}{\mathbb{R}}

\newcommand{\ex}{\mathbb{E}}

\def\og{\leavevmode\raise.3ex\hbox{$\scriptscriptstyle\langle\!\langle$~}}
\def\fg{\leavevmode\raise.3ex\hbox{~$\!\scriptscriptstyle\,\rangle\!\rangle$}}

\title{A Characterization of Wishart Processes and Wishart Distributions}

\author{Piotr Graczyk \thanks{LAREMA, Universit\'e d'Angers, France.piotr, graczyk@univ-angers.fr}, Jacek Ma\l{}ecki \thanks{Faculty of Pure and Applied Mathematics, Wroc{\l}aw University of Science and Technology, Poland, jacek.malecki@pwr.edu.pl. Supported by the National Science Centre Poland (2013/11/D/ST1/02622).}, and Eberhard Mayerhofer\thanks{University of Limerick, Castletroy, County Limerick, Ireland, eberhard.mayerhofer@ul.ie. Supported by ERC (278295) and SFI (08/SRC/FMC1389).}
}

\begin{document}
\maketitle
\begin{abstract}
A characterization of the existence of non-central Wishart distributions (with shape and non-centrality parameter) as well as the existence of solutions to Wishart stochastic differential equations (with initial data and drift parameter) in terms of their exact parameter domains is given. These two families are the natural extensions of the non-central chi-square distributions and the squared Bessel processes to the positive semidefinite matrices.
\end{abstract}
\thispagestyle{empty}

\newpage

\section{Introduction and Preliminaries}
The aim of this paper is to characterize the parameter domain of non-central Wishart distributions (with shape, scale and non-centrality parameters) and that of Wishart processes, a class of positive semi-definite diffusion processes (with drift parameter).

Denote by $\mathcal{S}_p$ the space of symmetric $p\times p$ matrices and let $\mathcal{S}_p^+$ be the open cone of positive definite matrices, with topological
closure $\bar {\mathcal S}^+_p$, the positive semi-definite matrices. The classical Gindikin\footnote{The name of this set originates from
Gindikin's \cite{bib:Gin} work in a general multivariate setting.} set $W_{0}$ is defined as the set of admissible $\beta\in\R$ such that there exists a random matrix $X$ with values in $\bar {\mathcal S}^+_p$ (equivalently a measure with support in $\bar {\mathcal S}^+_p$) such that its Laplace transform is of the form
$$
\ex e^{-\tr(uX)}=(\det(I+\Sigma u))^{-\beta},\ \ u\in \bar {\mathcal S}^+_p,
$$
where $\Sigma\in {\mathcal S}_p^+$. It is well-known (cf.~\cite{bib:farautKOR}, pp. 137, 349) that
$$
W_0=\frac12 B \cup \left[\frac{p-1}2,\infty\right) \/,
$$
where $B=\{0,1,\cdots,p-2\}$. 

A more intricate question concerns the existence of non-central Wishart distributions, which in addition involves a parameter of non-centrality:

\begin{definition}\label{def wish}
The general non-central Wishart distribution $\Gamma_p(\beta,\omega;\Sigma)$ on $\bar{\mathcal{S}}_p^+$ is
defined (whenever it exists) by its Laplace transform
\begin{equation}\label{FLT Mayerhofer Wishart}
\mathcal L (\Gamma_p(\beta,\omega;\Sigma))(u)= \left(\det(I+\Sigma u)\right)^{-\beta}e^{-\tr(u(I+\Sigma
u)^{-1}\omega)},\quad u\in \bar{\mathcal S}_p^+
\end{equation}
where $\beta>0$ denotes its shape parameter, $\Sigma\in \mathcal S_p^{+}$ is the
scale parameter and the parameter of non-centrality equals
$\omega\in \bar{\mathcal{S}}_p^+$. 
\end{definition}

Random matrices $X$ verifying \eqref{FLT Mayerhofer Wishart} arise in statistics as estimators of the covariance matrix parameter $\Sigma$ of a normal population. In fact, 
for the random matrix
$$
X = \xi_1\xi_1^T+\ldots+\xi_n\xi_n^T=:q(\xi_1,\dots,\xi_n)\/,
$$
where for $i=1,\dots,n$, 	$\xi_{i}\sim \mathcal N_p(m_i,\Sigma/2)$ are independent, normally distributed column vectors in $\R^p$, the Laplace transform of $X$ is given by
the right-hand side of \eqref{FLT Mayerhofer Wishart} with $\beta =n/2$ and $\omega=q(m_1,\ldots,m_n)$, {see Johnson and Kotz \cite[Chap.38 (47), p.175]{john-kotz}.}

Accordingly, the pair $(\omega,\beta)$ is said to belong to the non-central Gindikin set $W$
if there exists a random matrix $X$ with values in $\bar {\mathcal S}_p^+$ having the Laplace transform \eqref{FLT Mayerhofer Wishart} for a matrix $\Sigma\in {\mathcal S}_p^+$ \footnote{If $\Sigma$ is of maximal rank, this definition is indeed independent of $\Sigma$, see Lemma \ref{mayx}.}.

\newpage
Note the following:
\begin{itemize}
\item If $(\omega,\beta)\in W$ then $\beta\geq 0$, otherwise $\mathbb E[e^{-\tr(uX)} ]$ would be unbounded; and clearly,
$(0,\beta)\in W$ if and only if $\beta \in W_0$. 
\item In the case, where $\rank(\omega)=1$ and $\beta\not=0$, the characterization of the non-central Gindikin set $W$ is given in \cite{bib:PR}: then $(\omega,\beta)\in W$ if and only if
$\beta\in W_0$.
\item For $\beta>\frac{p-1}{2}$, Bru \cite{bib:b91} shows that Wishart  processes have Laplace transform given by \eqref{FLT  Mayerhofer Wishart}.
\end{itemize}

The general problem of existence and non-existence of non-central Wishart distributions is studied by Letac and Massam \cite{bib:LetMassFalse}\footnote{However, the statement and proof in \cite{bib:LetMassFalse} are incomplete, as pointed out by \cite{bib:mayer} and \cite{bib:mayerJMA}.}.
 In a more recent work Mayerhofer \cite{bib:mayerJMA} reveals that there is an interplay between the rank of the non-centrality parameter $\omega$ and the magnitude of $\beta$ in the discrete part of the classical Gindikin ensemble: if $(\omega,\beta)\in W$ and $2\beta\in B$, then $\rank(\omega)\le 2\beta +1$.

{The Laplace transform formulas in Johnson and Kotz \cite{john-kotz} and Bru \cite{bib:b91} and }the results in \cite{bib:mayerJMA} allow to conjecture\footnote{In \cite{bib:LetMass} and a previous version of this paper, the name {\it Mayerhofer Conjecture} is used. 
The conjecture was first presented at the CIMPA Workshop in Hammamet in 2011.} the following:
\medskip

{\bf NCGS Conjecture}. {\it The non-central Gindikin set is characterized by }
$$
(\omega,\beta )\in W \ \ \Leftrightarrow \ \ (2\beta\in [p-1,\infty), \;\omega\in \bar {\mathcal S}_p^+) \ {\it or}\ (2\beta\in B, \rank(\omega)\le 2\beta).
$$
 
 A proof of the NCGS Conjecture has been put forward by the preprint \cite{bib:LetMass}. The proof of \cite{bib:LetMass} is technical\footnote{It requires a detailed analysis of the singular and continuous part of certain non-central distributions. Besides, the present version of \cite{bib:LetMass} does
not prove that $(w, p) \in W$ implies $(0, p) \in W$.} and does not
provide an intuitive explanation for the particular parametric restrictions of shape and non-centrality parameter.

The present paper gives a first complete proof of the NCGS conjecture, which reveals and builds on the intimate connection between non-central Wishart distributions and Wishart processes (\cite{bib:b91}, see also \cite[Theorem 1.1]{donati2004some}). The latter constitute positive semi-definite solutions $(X_t)_{t\geq0}$ of stochastic differential equations of the form

\begin{eqnarray}\label{eq:Wishart:SDePLUS}
dX_t = \sqrt{X_t}dW_t+dW^T_t\sqrt{X_t}+\alpha Idt\/,\quad X_t\in \bar {\mathcal S}_p^+\/,\quad X_0=x_0\in \bar {\mathcal S}_p^+,
\end{eqnarray}
where $\sqrt{X_t}$ is the unique positive square root of $X_t$, $W$ is a $p\times p$ matrix of standard Brownian motions, and $\alpha\geq0$ is the drift parameter. 

Wishart processes are natural generalizations of squared Bessel Processes \cite{YorEnc}. It is demonstrated in the present paper that the existence of Wishart processes depends crucially on the drift parameter.

The paper proves a necessary and sufficient condition for the existence of Wishart processes, and how this existence issue is related to the one of Wishart distributions. Already Bru \cite{bib:b91}, who introduces Wishart processes for the first time, realizes the explicit formula for the Laplace transform of $X_t$:
\begin{proposition}Bru(\cite[Theorem 3] {bib:b91})\label{stochWallach}
If the stochastic differential equation (\ref{eq:Wishart:SDePLUS}) with $x_0\in \bar {\mathcal S}_p^+$ has a global weak solution in $\bar {\mathcal S}_p^+$, then 
$X_t$ is Wishart distributed for each $t\geq 0$. In particular, 
\begin{equation}
\label{Lap_Wish} 
\ex^{x_0}[\exp(- \tr(uX_t)]=(\det(I+2tu))^{-\alpha/2} \exp[- \tr(x_0(I+2tu)^{-1}u))],\quad u\in \bar{\mathcal S}^+_p\/.
\end{equation}
\end{proposition}

In the present paper, it is also shown how to construct full-fledged Wishart processes from individual Wishart distributions. The main result is thus a three-fold characterization:
\begin{theorem}\label{th super}
Let $x_0\in\bar{\mathcal S}_p^+$ and $\alpha\geq 0$. The following
are equivalent:
\begin{enumerate}
\item \label{super a} The SDE \eqref{eq:Wishart:SDePLUS} has a global weak solution with $X_0=x_0$.
\item \label{super b} Either $\alpha\geq p-1$, or $\alpha\in B$ and $\rank(x_0)\leq \alpha$.
\item \label{super c} $(x_0,\alpha/2)\in W$.
\end{enumerate}

\end{theorem}

Our proof of the NCGS Conjecture (that is, Theorem \ref{th super} \ref{super b} $\Leftrightarrow$ \ref{super c}) is based on an analysis of 
affine Wishart semigroups. As a new tool, the action of a class of polynomials on Wishart processes is used, which arise as coefficients of the characteristic polynomial of a symmetric matrix. A full characterization of Wishart processes is provided by (Theorem \ref{th super} \ref{super a} $\Leftrightarrow$ \ref{super b}). 

For convenience of the reader, but at the expense of proving an additional implication, Theorem \ref{th super} is split into
two independent theorems in the following two chapters. They require different mathematical tools and therefore
can be read independently. Chapter \ref{sec2} is concerned with the existence of solutions to Wishart stochastic differential equations
using elementary stochastic analysis with symmetric polynomials (Theorem \ref{char sdes} comprises the equivalence \ref{super a} $\Leftrightarrow$ \ref{super b} of Theorem \ref{th super}). Chapter \ref{sec em} concerns the existence of Wishart distributions (the NCGS conjecture, which comprises \ref{super b} $\Leftrightarrow$ \ref{super c} of Theorem \ref{th super} ). Here the Markovian viewpoint is used, in particular the fact that Wishart semigroups are affine Feller semigroups. Finally, in Section \ref{char low rank} a conjecture by Damir Filipovi\'c \cite{bib:filcon} on the existence of such semigroups on the cones of lower rank matrices is proved.
\section{Gindikin sets for Wishart Processes}\label{sec2}
This section studies the question of solutions in $\bar{\mathcal{S}}_p^{+}$ of the Wishart SDE \eqref{eq:Wishart:SDePLUS}, using the dynamics of some polynomial functionals of these solutions.

For a symmetric $p\times p$ matrix $X$, define  the  elementary symmetric polynomials 
 \begin{equation}\label{bijection}
   e_n(X) = \sum_{i_1<\ldots<i_n}\lambda_{i_1}(X)\lambda_{i_2}(X)\ldots \lambda_{i_n}(X)\/,\ \ \ \quad n=1,\ldots,p,
 \end{equation}
in the eigenvalues $\lambda_1(X) \le  \ldots\le\lambda_p(X)$ of $X$. Moreover, the convention $e_0(X)\equiv 1$ is used. Up to the sign change, the polynomials $e_n$ are the coefficients of the characteristic polynomial of $X$, i.e.
$$
\det(X-uI)=(-1)^p  u^p + (-1)^{p-1} e_1(X)u^{p-1}+\ldots  -e_{p-1}(X)u+e_p(X) 
$$
and are polynomial functions of the entries of the matrix $X$. In particular, $e_p(X)=\det X$. 

In \cite{bib:gm2},  symmetric polynomials related to general class of non-colliding particle systems were studied in details. Here  similar results are presented,
adapted to the matrix SDE
\begin{equation}\label{MatrixSDE}
  dX_t = g(X_t)dW_th(X_t)+h(X_t)dW_t^Tg(X_t)+b(X_t)dt\/,
\end{equation}
where the continuous functions $g,h,b$ act spectrally\footnote{
Recall that if $g:\R\mapsto\R$ then $g(X)$ is defined spectrally, i.e. $g(U \diag(\lambda_i) U^T)=U \diag(g(\lambda_i)) U^T$, where $U\in SO(p)$.} on $\mathcal S_p$ and $W_t$ is a Brownian $p\times p$ matrix. Henceforth, abbreviate $\sigma=2gh$ and $G(x,y) = g^2(x)h^2(y)+g^2(y)h^2(x)$.
Furthermore, the natural bijection \eqref{bijection} between the   eigenvalues  $\Lambda=(\lambda_1\ldots  \lambda_p)$ 
and the polynomials $e=(e_1,\ldots,e_p)$ is used,  extended  to the closed Weyl chamber $\bar C_+= \{(x_1,\ldots,x_p)\in\R^p: x_1\le x_2<\ldots\le x_p\}$, see  \cite[p.6]{bib:gm2}.
Furthermore, write $\Lambda=\Lambda(e)$ for the inverse bijection on the set $\overline{e(C_+)}$.
The notation $e_{n}^{\overline i}$ for the incomplete polynomial of order $n$,  not containing the variable $\lambda_i(e)$, is used; the notation $e_{n}^{\overline i,\overline j}$ is analogous. {Moreover, set $e_{0}^{\overline i}\equiv 1$ and $e_{-1}^{\overline i,\overline j}\equiv 0$.}
\medskip 
\begin{proposition}\label{Matrix_to_Polynomial} 
Let $X=(X_t)_{t\geq 0}$ be a weak solution of (\ref{MatrixSDE}) (with possible finite time blow-up). Then the symmetric polynomials 
$e_n=e_n(X)$, $n=1,\ldots, p$, are continuous semimartingales described   
by the system of SDEs ($n=1,\ldots,p$)
\begin{eqnarray} \label{eq:en:SDE1}
  de_n = \left(\sum_{i=1}^p\sigma^2(\lambda_i(e))(e_{n-1}^{\overline i})^2\right)^{\frac{1}{2}}dV_n 
+\left(\sum_{i=1}^pb(\lambda_i(e))e_{n-1}^{\overline i}-\sum_{i<j}G(\lambda_i(e),\lambda_j(e))e_{n-2}^{{\overline i},{\overline j}}\right)dt\/,
 \end{eqnarray}
 where $V_n$ are Brownian motions on $\R$ such that 
$
   d\left<e_n,e_m\right> = \sum_{i=1}^p \sigma^2(\lambda_i(e) )e_{n-1}^{\overline i}e_{m-1}^{\overline i}dt\/.
$
\end{proposition}
\begin{proof}
Note that here the equation is considered on $\mathcal{S}_p$ which does not require solutions to live in $\bar{\mathcal{S}}_p^+$ (as is required in reference to Wishart processes). Since the coefficients of the equation (\ref{MatrixSDE}) are continuous, a local weak solution exists.  This solution, before its possible blow-up, is considered.

The symmetric polynomials $(e_1,\ldots, e_p)$ are given by  an analytic function 
\[
F: \mathcal{S}_p\rightarrow \mathbb R^p,\quad  X\to (e_1(X),\ldots,e_p(X)),
\]
since each elementary symmetric polynomial is given in terms of the coefficients of the matrix $X$. {Thus It\^o's formula implies that $(e_1,\ldots, e_n)$ are continuous semimartingales (for every starting point $x_0$ and even when the eigenvalues collide).}

The SDEs describing $(e_1,\ldots, {e_p})$ can be determined similarly as in Propositions 3.1 and 3.2 in \cite{bib:gm2}, 
which generalize the proof of (4.1) in \cite{bib:b91}.
One uses the SDEs
 for the eigenvalues
\begin{eqnarray}
\label{eq:eigenvalues}
 d\lambda_i = 2g(\lambda_i)h(\lambda_i)dB_i+\left({b}(\lambda_i)+\sum_{j\neq i}\frac{G(\lambda_i,\lambda_i)}{\lambda_i-\lambda_i}\right)dt\/,\quad i=1,\ldots,p\/,
\end{eqnarray}
which are available, according to Theorem 3 from \cite{bib:gm11}, when  eigenvalues $\lambda_i(0)$ of $x_0$ are all distinct and before their eventual collision. {However, the It\^o formula states that the martingale part and the bounded variation part of $(e_1,\ldots, e_p)$ are given in terms of derivatives of the smooth function $F$ and those derivatives have
just be determined on the open set $U=\{X\in \mathcal{S}_p: \lambda_i(X)\neq \lambda_j(X) \text{ for all } i\neq j, \text{ with } 1\leq i,j\leq p\/\}$. Since the derivatives of $F$ are continuous on $\mathcal{S}_p$ as well as the coefficients in \eqref{eq:en:SDE1} (the singular expressions $(\lambda_i-\lambda_j)^{-1}$ appearing in (\ref{eq:eigenvalues}) are no longer present in \eqref{eq:en:SDE1}), one can conclude, by continuity, that the equalities hold on $\bar{U}=\mathcal{S}_p$, i.e. one can drop  the conditions that eigenvalues of the initial point are all different and that they are non-colliding for $t>0$.}

\end{proof}

Using Proposition \ref{Matrix_to_Polynomial} the following characterization of the symmetric polynomials related to Wishart processes is obtained:
\medskip
\begin{proposition}
\label{prop:Poly}
Let $X_t$ be a Wishart process, i.e. a solution of the matrix SDE \eqref{eq:Wishart:SDePLUS}. Then the
 symmetric polynomials $e_n=e_n(X)$, $n=1,\ldots, p$ are semimartingales satisfying the following system of SDEs
\begin{eqnarray}
d e_1&=&{2 \sqrt{e_1}dV_1+ p\alpha dt,}\label{eq:e1:SDEs}\\
   de_n &=& M_n(e_1,\ldots,e_p)dV_n +(p-n+1)(\alpha-n+1)e_{n-1}dt\/,\quad n=2,\ldots,p-1\/, \label{eq:polynom_first:SDEs} \\
   de_p &=& 2\sqrt{e_{p-1}e_p}dV_p +(\alpha-p+1)e_{p-1}dt,  \label{eq:polynom_last:SDEs}  
	\end{eqnarray}
where  $V_n$, $n=1,\ldots, p$  are  one-dimensional Brownian motions and the functions $M_n$ are continuous on $\R^p$.
Furthermore, for  $n=1,\ldots, p$, the processes ${\mathcal M}_n(t):=\int_0^t M_ndV_n$ are martingales satisfying
\begin{equation}\label{finite qvar}
\mathbb E[\int_0^t\langle {\mathcal M}_n, {\mathcal M}_n \rangle_s ds]<\infty, \quad \text{for each} \  t>0\ 
{ {\it and}
\ n=1,\dots, p.}
\end{equation}
	\end{proposition}
	
	\medskip
	\begin{remark}
 Note that by Proposition \ref{Matrix_to_Polynomial}, the explicit forms of the martingale parts  $ {d{\mathcal M}_n}=M_n(e_1,\ldots,e_p)dV_n$  as well as their brackets 
$d\left<e_n,e_{m}\right>$  are known   for every $n,m=1,\ldots,p$.
\\
 {Equation \eqref{eq:e1:SDEs} is given by Bru 
 	\cite{bib:b91} and is used in the proof of \eqref{finite qvar}}.
 Equation \eqref{eq:polynom_last:SDEs} is just kept for informative reasons. They are both covered by \eqref{eq:polynom_first:SDEs}, by setting $n=1$ and $n=p$.

\end{remark}
		\begin{proof} Applying Proposition \ref{Matrix_to_Polynomial} to the SDE \eqref{eq:Wishart:SDePLUS},  one finds that
\begin{equation}\label{eq: Mn}
	M_n=2\left(\sum_{i=1}^p\lambda_i(e_{n-1}^{\overline i})^2\right)^{1/2}.
\end{equation}
Moreover, the drift coefficients of $de_n$ satisfy
	$$ 
	\sum_{i=1}^p \alpha e_{n-1}^{\overline i}-\sum_{i<j}(\lambda_i+\lambda_j)e_{n-2}^{\overline{i},\overline j}=(p-n+1)(\alpha-n+1)e_{n-1}.
	$$

It remains to show \eqref{finite qvar}, for each $n=1,\dots,p$. For $n=1$, {by \eqref{eq:e1:SDEs}},
 $e_1(t)$ is a squared Bessel process. Furthermore, since $e_1(t)$ is non-centrally chi-squared distributed, for each $t>0$, and for each $m\geq 1$ 
\begin{equation}\label{eq M1}
\int_0^t \mathbb E[\vert e_1(s)\vert^m ds]<\infty,
\end{equation}
hence by Fubini
\begin{equation*}
 \mathbb E[\int_0^t\vert e_1(s)\vert^m ds]<\infty.
\end{equation*}
For $m=1$, this estimate implies
\begin{equation}
\mathbb E[\int_0^t\langle \mathcal M_1, \mathcal M_1 \rangle_s ds]<\infty,
\end{equation} for each $t>0$. For $1<n\leq p$ one can use \eqref{eq: Mn} to obtain the estimate
\[
\langle \mathcal M_n,\mathcal M_n\rangle= 4\sum_{i=1}^p\lambda_i(t)  (e_{n-1}^{\bar i})^2\leq  4 e_{n-1}(t)\leq 4 e_1 ^{2n-2}(t),
\]
and thus, by \eqref{eq M1},  one obtains \eqref{finite qvar}.
	\end{proof}
Since a Wishart process is $\bar{\mathcal{S}}_p^+$ valued by definition, so $e_n\geq 0$, for all $n=1,\dots,p$. The idea
of the proof of the next Theorem is to show that for $(x_0,\beta)\not \in W$, some of the symmetric polynomials $e_n$ become strictly negative.

\subsection{Solving the Wishart stochastic differential equations}

This section gives a full characterization of the existence of solutions to Wishart SDEs \eqref{eq:Wishart:SDePLUS}.

\begin{theorem}\label{char sdes}
Let $\alpha\geq 0$, and $x_0\in \bar{\mathcal S}_p^+$. The following are equivalent.
\begin{enumerate}
\item \label{ax} The SDE \eqref{eq:Wishart:SDePLUS} has a global weak solution with $X_0=x_0$.
\item \label{bx} $\alpha\geq p-1$, or $\alpha\in \{0,1,\dots, p-2\}$ and $\rank(x_0)\leq \alpha$.
\end{enumerate}
\end{theorem}

\begin{proof}

Assume first \ref{ax}. If $\alpha\geq p-1$, nothing has to be shown. Suppose, therefore, $\alpha<p-1$. 
Recall {formulas  \eqref{eq:e1:SDEs}--\eqref{finite qvar}}
 from Proposition \ref{prop:Poly}. One can compute explicitly the expected value of the polynomials starting from the first one,
\begin{eqnarray*}
\ex e_1(t) = e_1(0)+p\alpha \int_0^t ds = e_1(0)+p\alpha t.
\end{eqnarray*}
Therefore
\begin{eqnarray*}
\ex e_2(t) &=& e_2(0)+(p-1)(\alpha-1)\int_0^t \ex e_1(s)ds\\
&=& e_2(0)+(p-1)(\alpha-1)e_1(0)t+p(p-1)\alpha(\alpha-1)\frac{t^2}{2},
\end{eqnarray*}
and so on. Consequently $\ex e_n(t)$ is a polynomial of degree not greater than $n$. In particular, the coefficient of $t^n$ is
\begin{eqnarray*}
\frac{p(p-1)\cdot\ldots\cdot(p-n+1)\cdot\alpha(\alpha-1)\cdot\ldots\cdot (\alpha-n+1)}{n!}.
\end{eqnarray*}
If $\alpha\notin B$ and $n$ is the first integer greater than or equal to $\alpha+1$, then $\ex e_n(t)$ is a polynomial of degree $n$ such that the leading coefficient is negative. Consequently, it cannot stay positive for every $t>0$, which is an impossibility.

\medskip

If $\alpha=m\in B$, consider $\ex e_n(t)$ where $n=m+1$. Then
\begin{eqnarray*}
\ex e_n(t) = e_n(0)+(p-n+1)(\alpha-n+1)\int_0^t \ex e_{n-1}(s)ds = e_n(0).
\end{eqnarray*}
If $e_n(0)> 0$, then 
\begin{eqnarray*}
\ex e_{n+1}(t) = e_{n+1}(0)+(p-n)(\alpha-n)e_n(0)t,
\end{eqnarray*} 
i.e. the leading term is negative and thus $\ex e_{n+1}(t)<0$ for large $t$. It implies $e_n(0)=0$, i.e. $\rank(x_0)\leq n-1=m=\alpha$. 

Proof of \ref{bx} $\Rightarrow$ \ref{ax}:

The existence of global weak solutions 
 for $\alpha\geq p-1$ is  { proved by 
Bru \cite{bib:b91} (Bru's  proof for $\alpha > p-1$   can be easily 
	extended to  $\alpha \ge p-1$) 
and in 
\cite[Theorem 2.6]{bib:CFMT}.} 
Therefore, only the cases $\alpha\in \{0,1,\dots,p-2\}$ need to be considered. If $\alpha=0$, then $X=0$
is the global weak solution of \eqref{eq:Wishart:SDePLUS}, for initial value $x_0=0$. Let therefore $1\leq \alpha\leq p-2$, and $\rank(x_0)\leq \alpha$. Let $B_1,B_2,\dots, B_\alpha$ be a sequence of independent, $p$--dimensional standard Brownian motions, and let $y_1,\dots, y_\alpha\in\mathbb R^p$ {be} such that $x_0=
y_1 y_1^\top+\dots y_\alpha y_\alpha^\top$. Then the process
\[
X_t:=\sum_{i=1}^\alpha(y_i+B_i)(y_i+B_i)^\top
\]
is a continuous semimartingale, by construction, and $X_0=x_0$. Furthermore $dX_t=dM_t+\alpha I dt$,
where $I$ is the $p\times p$ unit matrix, and $(M_t)_t$ is a continuous martingale having quadratic
variation \eqref{eq: qvar}. Therefore, by Proposition \ref{prop wish semi}, the Wishart SDE \eqref{eq:Wishart:SDePLUS}
has a global weak solution.
\end{proof}

\begin{remark}
	 Necessity of {\rm (ii)} can be also proved, if the validity of the NCGS Conjecture is assumed (a fact that is proven in Section \ref{sec em}, and which has not been used above to keep the section self-contained). Suppose the existence of a weak solution. Then by Proposition \ref{stochWallach}, the solution is Wishart distributed,
that is, for each $t\geq 0$, $X_t\sim\Gamma_p(\alpha/2,x_0;2t I)$. By the NCGS Conjecture, $\alpha/2\in W_0$ and, in addition, if $\alpha<p-1$ then $\rank(x_0)\leq \alpha$.
\end{remark}
\section{The NCGS Conjecture and Wishart Semigroups}\label{sec em}
In this section Wishart semigroups are introduced, which are the main tool for the proof of the NCGS Conjecture in Section \ref{proof NCGS} below. In Section \ref{char low rank} all Wishart semigroups on lower rank matrices are characterized.
\subsection{Wishart semigroups}
For $p\geq 1$, let $D_p(k)\subset \bar{\mathcal S}_p^+$ be the sub-cones of rank $\leq k$ matrices, 
$0\leq k\leq p$, where clearly $D_p(0)=\{0\}$ and $D_p(p)=\bar{\mathcal S}_p^+$. Denote by
$f_u(x)=\exp(\tr(-ux))$, where $u,x\in \bar{\mathcal S}_p^+$.
\begin{definition}
Let $D\subset \bar{\mathcal S}_p^+$ be a closed set. A Wishart semigroup $(P_t)_{t\geq 0}$ on $D$ is a positive, strongly continuous $C_0(D)$ contraction semigroup which for any $u\in\mathcal S_p^+$ acts on
$f_u\mid_D$ as
\begin{equation}\label{eq:action P}
P_t f_u(x)=\det(I+2tu)^{-\alpha/2} e^{-\tr(x (u^{-1}+2 tI)^{-1})}, \quad x\in D.
\end{equation}
Here $\alpha\geq 0$ is called the drift parameter of $(P_t)_{t\geq 0}$.
\end{definition}

Note: A Wishart semigroup may or may not exist, depending on the choice of $\alpha$ and $D$. In Theorem
\ref{nude semigroups} below, the existence of Wishart semigroups for $D=D_p(k)$ is characterized.

The following remark summarizes several essential properties of Wishart semigroups:
\begin{remark}\label{rem1x}
Let $(P_t)_{t\geq 0}$ be a Wishart semigroup with drift parameter $\alpha$.
\begin{enumerate}
\item \label{rem1x1} (Markovian representation)
In view of the Riesz representation theorem for positive functionals \cite[Chapter 2.14]{Rudin}, for each $t>0$, $x\in D$
there exists a positive measure $p_t(x,d\xi)$ on $D$ such that
\begin{equation}\label{eq: action P}
P_t f(x)=\int _D f(\xi)p_t(x,d\xi).
\end{equation}
Furthermore, the semigroup property of $(P_t)_{t\geq 0}$ implies, that $p_t(x,d\xi)$ satisfies
the Chapman-Kolmogorov equations, thus $p_t(x,d\xi)$ is a Markov transition function. Hence,
the semigroup has a stochastic representation as a Markov process $(\mathbb P^x)_{x\in D}$,
where for each $x\in D$, $\mathbb P^x$ denotes the resulting probability on the canonical path space $D^{\mathbb R_+}$ with initial law $\delta_x$, and $X_t(\omega):=\omega(t)$, where $\omega\in D^{\mathbb R_+}$.
{ The Markov process $(X, \mathbb P^x)$ is called the {canonical representation} of the semigroup $ (P_t)_{t\geq 0}$ .}
\item (C\`adl\`ag Paths) It is a well-established fact, that any Feller process (that is, a Markov process with strongly continuous $C_0$ semigroup),
has a c\`adl\`ag version. 
\item (Affine Property) By definition, Wishart semigroups are affine semigroups (see \cite{bib:CFMT}), that is, the Laplace transform of their transition function
is of the form
\begin{equation}\label{ap}
\mathbb E[e^{-\tr(u X_t)}\mid X_0=x]=e^{-\phi(t,u)-\tr(\psi(t,u) x)},
\end{equation}
where
\[
\phi(t,u)=\frac{\alpha}{2} \log(\det(I+2t u)),\quad \psi(t,u)=(u^{-1}+2tI)^{-1}.
\]
\item (Wishart transition function) By definition, the Markovian transition function of a Wishart semigroup $p_t(x,d\xi)$ is $\Gamma_p(\alpha/2,x ; 2tI)$ distributed, for each $t\geq 0$ and for all $x\in D$. Furthermore,
the support of $\Gamma_p(\alpha/2,x ; 2tI)$ is contained in $D$.
\item (Non-Explosion)
$(P_t)_{t\geq 0}$ is conservative: Let $u_n\in\mathcal {S}_p^+$
such that $u_n\rightarrow 0$ as $n\rightarrow \infty$. By \eqref{eq: action P} and Lebesgue's dominated convergence theorem one thus has
\[
P_t 1=\lim _{n\rightarrow\infty}P_t f_{u_n}(x)=1.
\]
\item (Semimartingales and Continuity) If, in addition, one assumes that the linear span of $D$ has non-empty interior, $(X,\mathbb P_x)$ for each $x$ is an affine semimartingale, that is, a semimartingale with differential characteristics which are affine functions in the state variable.
{The continuity of the sample paths of $X$ follows}. For more details, see Appendix \ref{appendix a}.
\item (Strong Maximum Principle)
For a strongly continuous $C_0$ semigroup $(P_t)_{t\geq 0}$ with infinitesimal generator $\mathcal A$, the following are equivalent
\begin{enumerate}
\item \label{first im} $\mathcal A$ satisfies the strong maximum principle, that is, $\mathcal Af(x_0)\geq 0$, for any $f\in C_0$ that satisfies $f(x)\geq f(x_0)$.
\item \label{se im} $(P_t)_{t\geq0}$ is positive (hence a Feller semigroup).
\end{enumerate}
The proof of {(a) $\Rightarrow$ (b)} is simple. A proof of the non-trivial implication (b) $\Rightarrow$ (a) employs the positivity of the Yoshida approximations of $\mathcal A$ (\cite[Corollary 2.8]{Kurtz}).
\end{enumerate}
\end{remark}
Wishart semigroups on $D=\bar{\mathcal S}_p^+$ are well understood; they are the semigroups associated with affine diffusion processes on $D$.  By \cite[Theorem 2.4]{bib:CFMT} the following are equivalent:
\begin{itemize}
\item The Wishart semigroup with drift parameter $\alpha$ exists with state space $D=\bar{\mathcal S}_p^+$.
\item $\alpha\geq p-1$.
\end{itemize}
However, for strict subsets $D\subset\bar{\mathcal S}_p^+$, less is known about Wishart semigroups. In Theorem \ref{nude semigroups}
below a new result for the sets of rank $k\leq p-1$ matrices is given.

Let $\mathcal S^*_p$ be the space of rapidly decreasing smooth functions on $\mathcal S_p$, and
for a subset $D\subseteq \bar{\mathcal S}_p^+$, let
$S^*_p(D)=\{f\mid_{D}\mid f\in \mathcal S_p^*\}$.

For any $f\in \mathcal S^*_p(D)$, the action of the following differential operator
is well-defined,
\begin{equation}\label{A sharp}
\mathcal A^\sharp f(x)=2\tr(x\nabla^2 )f(x)+\alpha \tr(\nabla f(x)),
\end{equation}
where the notation of Bru \cite{bib:b91}
\[
x\nabla^2:=x\cdot \nabla\cdot \nabla
\]
is used, with $\cdot$ denoting the matrix multiplication, and $\nabla$ being the matrix of partial differential operators
$\nabla=(\nabla_{ij})_{ij}$, where $\nabla_{ij}=\frac{\partial}{\partial x_{ij}}$. This expression
reads in canonical coordinates, (cf. the notation of \cite[Theorem 2.4]{bib:CFMT})
\[
2\tr(x\nabla^2)=\sum_{i,j,k,l} A(x)_{i,j,k,l}\frac{\partial ^2}{\partial x_{ij}\partial x_{kl}},
\]
where $A(x)$ is a quadratic form on $(\bar{\mathcal S}_p^+)^2$, defined in coordinates as
\[
A(x)_{i,j,k,l}=x_{ik}\delta_{jl}+x_{il}\delta_{jk}+x_{jk}\delta_{il}+x_{jl}\delta_{ik}.
\]

\begin{proposition}\label{prop genx}
Suppose $D_p(1)\subseteq D$, and let $(P_t)_{t\geq 0}$ be a Wishart semigroup on $D$ with infinitesimal generator $\mathcal A$. Then $S^*_p(D)\subset \mathcal D(\mathcal A)$ and $\mathcal Af=\mathcal A^\sharp f$ in \eqref{A sharp} for any $f\in S^*_p(D)$.
\end{proposition}
\begin{proof}
It is first proved that
\begin{equation}\label{action fu}
\mathcal Af^D_u=(\mathcal A^\sharp f_u)\mid_D,
\end{equation}
for any exponential $f_u^D(\cdot):=e^{-\tr(u\cdot)}\mid_{D}$. Here the right hand side
involves differentiation on the open domain $\mathcal S_p$, and later restriction to $D$,
whereas on the left hand side $\mathcal A$ acts directly on 
{$f_u^D$}.

By the definition of the affine property \eqref{ap}, 
\begin{equation}\label{gen form 2}
\mathcal A f_u(x)=(F(u)+\text{tr}(R(u)x)f_u(x), \quad x\in D,
\end{equation}
for $f_u(x)=\exp(-\text{tr}(ux))$ and $u\in \bar{\mathcal S}_p^+$, and thus ${f_u^D}\in\mathcal D(\mathcal A)$.
Here 
\begin{equation}\label{eq F}
F(u)=\frac{\partial\phi(t,u)}{\partial t}|_{t=0}=\alpha\tr(u)
\end{equation}
and
\begin{equation}\label{eq R}
R(u)=\frac{\partial \psi(t,u)}{\partial t}|_{t=0}=-2 u^2, 
\end{equation}
where the differentiation rules for
inverse map and determinant (\cite[Proposition  III.4.2 (ii) and Proposition II.3.3 (i)]{bib:farautKOR}) have been used. The assumption that $D$ contains rank one matrices implies that the convex hull of $D$ equals $\bar{\mathcal S} _p^+$, and thus $F$ and $R$ are uniquely determined, as the coefficients of the affine (in the state variable $x$) function
\[
x\mapsto F(u)+\tr(x R(u)).
\]
A straightforward computation reveals that the action of $\mathcal A^\sharp$ on $f_u^D$
coincides with \eqref{gen form 2}, hence \eqref{action fu} holds.

According to the density argument \cite[Theorem B.3]{bib:CFMT}, the linear hull of such exponentials
for strictly positive definite $u$ is dense in the space of rapidly decreasing functions on $\bar{\mathcal S}_p^+$ and thus equality in \eqref{action fu} extends, by convergence  properties in the Schwarz class and the closedness of $\mathcal A$, to rapidly decreasing functons.
\end{proof}
Recall that a time-homogenous Markov process is polynomial if the action of its semigroup can be extended to polynomials of any order (\cite[Definition 2.1]{Cuchiero11}).
\begin{proposition}\label{prop GM1}
Suppose $(P_t)_{t\geq 0}$ is a Wishart semigroup supported on $D\subset \bar{\mathcal S}_p^+$ with drift $\alpha\geq 0$. $(P_t)_{t\geq 0}$ is polynomial and its infinitesimal generator acts on symmetric polynomials as follows
\begin{equation}\label{gm eq1}
\mathcal Ae_n(x)=(p-n+1)(\alpha-n+1)e_{n-1}(x),\quad x\in D,\quad 1\leq n\leq p.
\end{equation}
\end{proposition}
\begin{proof}
By Proposition \ref{prop a22}, there is a version $(\widetilde X_t)_{t\geq 0}$  of $(X_t)_{t\geq 0}$ which is a Wishart semimartingale, and thus by Proposition \ref{prop wish semi} there exists a $d\times d$ dimensional Brownian motion $W$ such that the pair ($(\widetilde X_t)_{t\geq 0}, W$) constitutes a global weak solution of the Wishart SDE. Hence Proposition \ref{prop:Poly}
may be applied, that yields the SDE dynamics \eqref{eq:e1:SDEs}--\eqref{eq:polynom_last:SDEs}. By \eqref{finite qvar}, {$\int_0^t M_n dV_{n}$}
are true martingales, hence 
\[
\mathbb E^x[e_n(t)]=e_n(x)+(p-n+1)(\alpha-p+1)\mathbb E^x[\int_0^t e_{n-1}(s)ds],
\]
thus by Lebesgue's dominated convergence theorem,  
\[
\mathcal A e_n(x)=\lim_{t\downarrow 0}\frac{P_t e_n(x)-e_n(x)}{t}=(p-n+1)(\alpha-p+1)e_{n-1}(x).
\]

\end{proof}

An equivalence relation $\simeq$ on the space of random variables with values in $\bar{\mathcal S}_p^+$ is introduced by defining
$X\simeq Y$ if and only if for all $0\leq r\leq p$
\begin{center}
 $\mathbb P[\rank(X)=r]>0$ if and only if $\mathbb P[\rank(Y)=r]>0$  .
\end{center}

Three technical lemmas are useful:
\begin{lemma}\label{mayx}
Let $\beta\geq 0,\omega,\in \bar{\mathcal S}_p^+$ and $\Sigma\in\mathcal S_p^+$.
\begin{enumerate}
\item \label{x1} (linear automorphism) Let $\Sigma=q q^\top$, where $q$ is a real $p\times p$ matrix. If $X\sim\Gamma_p(\beta,\omega; I)$, then $Y=qXq^\top\sim \Gamma_p(\beta,q\omega q^\top; \Sigma)$ and $Y\simeq X$.
Conversely, $Y\sim\Gamma_p(\beta,q\omega q^\top; \Sigma)$ implies $X=q^{-1}Y (q^{-1})^\top\sim\Gamma_p(\beta,\omega;I)$.
\item \label{x2} (exponential family) If $X\sim\mu(d\xi)\sim\Gamma_p(\beta,\omega;I)$, then
for $v:=\Sigma^{-1}-I$ there exists a random variable $Y$ distributed as
\[
Y\sim \frac{\exp(\tr(v\xi))\mu(d\xi)}{\mathbb E[\exp(\tr(vX))]} \sim \Gamma_p(\beta,\Sigma\omega\Sigma;\Sigma)
\]\
and $Y\simeq X$. Conversely,  $Y\sim\Gamma_p(\beta,\Sigma\omega\Sigma;\Sigma)$ implies
that $X\sim\Gamma_p(\beta,\omega; I)$
\item \label{x3} If $X\sim\Gamma_p(\beta,\omega;\Sigma)$ then $\Gamma_p(\beta,{\tilde \omega; \tilde\Sigma})$ exists for any ${\tilde\omega}$ satisfying $\rank({\tilde\omega})\leq \rank(\omega)$ and for any ${\tilde\Sigma}\in\bar{\mathcal S}_p^+$.
\end{enumerate}
\end{lemma}
\begin{proof}
The equivalence relation in \ref{x1} holds, since any linear automorphism maintains the rank of matrices. 
The remaining claims in \ref{x1} follow from the following chain of identities, using the very definition of the Wishart distribution in terms of
its Laplace transform (using multiplicativity of the determinant and the cyclic property of the trace):
\begin{align*}
\mathbb E[e^{-\tr(u Y)}]&=\mathbb E[e^{-tr(u q X q^\top)}]=\mathbb E[e^{-\tr((q^\top u q)X)}]=(\det(I+ q^\top u q))^{-\beta}e^{-\tr(q^\top u q (I+q^\top u q)^{-1}\omega)}\\
&=(\det(I+\Sigma u))^{-\beta}e^{u (I+\Sigma u)^{-1} q\omega q^\top},
\end{align*}
i.e. $Y\sim\Gamma_p(\beta,q \omega q^\top; \Sigma)$.

Proof of \ref{x2}:
 Note that due to Proposition \ref{FLT maximal}, $v=-I+\Sigma^{-1}\in D(\mu)$ {\and} and \eqref{FLT Mayerhofer Wishart} holds for $v$. Hence the first part of the proof of \ref{x2} follows the lines of the proof of  \cite[Proposition 3.1 (ii)]{bib:mayerJMA}. Conversely, let $Y\sim\mu_1=\Gamma_p(\beta,\Sigma\omega\Sigma; \Sigma)$. Then $v_1=-\Sigma^{-1}+I\in D(\mu_1)$ and, after a few computations, one obtains
\[
\int e^{-\tr((u+v_1)\xi)}\mu_1(d\xi)=\left((\det(\Sigma))^{-\beta}e^{-\tr((\Sigma-I)\omega)}\right)(\det(I+u))^{-\beta} e^{-\tr(u(I+u)^{-1}\omega)},
\]
where the pre-factor is recognized as
\[
(\det(\Sigma))^{-\beta}e^{-\tr((\Sigma-I)\omega)}=\mathbb E[e^{-\tr(v_1 Y)}],
\]
and the second factor equals
\[
(\det(I+u))^{-\beta} e^{-\tr(u(I+u)^{-1}\omega)}=\mathbb E[e^{-\tr(uX)}]
\]
for $X\sim\Gamma_p(\beta,\omega;I)$.

Finally, for any ${ u}\in -\Sigma^{-1}+\mathcal S_p^+$, let
\[
\nu(d\xi):=\frac{\exp(-\tr({ u}\xi))\mu(d\xi)}{\mathbb E[\exp(-\tr(({ u}X))]}.
\]
Then $\nu(B)>0$ if and only if $\mu(B)>0$, for any Borel set $B\subset\bar{\mathcal S}_p^+$. Hence $Y\simeq X$ in \ref{x2}.

Proof of \ref{x3}: Let $\rank(\omega)=r$ with $0\leq r\leq p$. The following outlines the transformations that map
$\Gamma_p(\beta,\omega;\Sigma)$ onto $\Gamma_p(\beta,\omega_1;\Sigma_1)$. 

Suppose first $\rank(\omega_1)=r$ and that $\Sigma_1=q_1q_1^\top$ is of full rank. By properties
of the Natural Exponential Family \ref{x2}, one obtains $\Gamma_p(\beta,\Sigma^{-1}\omega \Sigma^{-1}; I)$. By \ref{x1} the transformation $\xi\mapsto q_a\xi q_a^{\top}$, where $q_a$ is an invertible but not necessarily symmetric matrix, yields $\Gamma_p(\beta,q_a\Sigma^{-1}\omega \Sigma^{-1} q_a^\top; \Sigma_a)$, where
$\Sigma_a:=q_a q_a^\top$. Again using \ref{x2} yields
\[
\Gamma_p(\beta,\Sigma_a^{-1}q_a\Sigma^{-1}\omega \Sigma^{-1} q_a^\top\Sigma_a^{-1}; I)=\Gamma_p(\beta,(q_a^{-1})^\top\Sigma^{-1}\omega \Sigma^{-1} q_a^{-1}; I)
\]
Finally, by \ref{x1}, the linear transformation $\xi\mapsto q_1 \xi q_1^{\top}$ yields
\[
\Gamma_p(\beta,q_1(q_a^{-1})^\top\Sigma^{-1}\omega \Sigma^{-1} q_a^{-1}q_1^\top; \Sigma_1)
\]
Note that $q_a$ has not been specified yet. Since the linear automorphism group acts transitively on $\bar{\mathcal S}_p^+$ and maintains ranks, there exists $q_a$
such that
\[
q_1(q_a^{-1})^\top\Sigma^{-1}\omega \Sigma^{-1} q_a^{-1}q_1^\top=\omega_1,
\]
and thus one obtains the existence of $\Gamma_p(\beta,\omega_1;\Sigma_1)$ for any invertible $\Sigma_1$ and any $\omega_1$ with $\rank(\omega_1)=r$.

Finally, let $\rank(\widetilde\omega)\leq \rank(\omega)=r$ and let $\widetilde\Sigma$ be not necessarily invertible.  Let $(\omega_n)_n$ be a sequence of non-centrality parameters $\omega_n$ such that 
$\lim_{n\rightarrow\infty}\omega_n=\widetilde\omega$, where $\rank(\omega_n)=r$ for each $n$, and let $(\Sigma_n)_n$ be a sequence of non-singular matrices $\Sigma_n$ such that $\lim_{n\rightarrow\infty }\Sigma_n= \widetilde\Sigma$.

By the previous arguments, 
\[
\Gamma_p(\beta,\omega_n;\Sigma_n)
\]
exists for any $n\in\mathbb N$. By Proposition \ref{FLT maximal}, for each $n$, the characteristic
functions are of the same form, and converge for any $u\in i\mathcal S_p$ as $n\rightarrow \infty$ to 
\[
\left(\det(I+\widetilde\Sigma u)\right)^{-\beta}e^{-\tr(u(I+\widetilde\Sigma
u)^{-1}\widetilde\omega)}
\] 
Hence, by L\'evy's continuity theorem, the limit is the characteristic function of a positive measure on $\bar{\mathcal S}_p^+$, namely $\Gamma_p(\beta,\widetilde\omega;\widetilde\Sigma)$.
\end{proof}
\begin{lemma}\label{extra supp}
Let $\Xi$ be a positive semi-definite random matrix supported on $D_p(r-1)$ and $\rank(\Xi)=r-1$ with nonzero probability, where $1\leq r\leq { p}$.
Let further $\eta\sim \mathcal N(\mu,\Sigma)$ with $\mu\in \mathbb R^p$ and with covariance matrix $\Sigma\in \mathcal S_p^+$. If $\Xi$ and $\eta$ are independent, then
$\rank(\Xi+\eta \eta^\top)=r$ with nonzero probability.
\end{lemma}
\begin{proof}
Assume first the constant case $\Xi=\Xi_0\in \bar{\mathcal S}_p^+$. Without loss of generality, one may assume $\Xi_0=\diag(I_{r-1},0)$, where $I_k$ is the $k\times k$ unit matrix. Define
\[
V=\left(\begin{array}{ll} I_{r-1} & -\Omega\\ 0& I_{p-r+1}\end{array}\right)
\]
with a $(r-1)\times (p-r+1)$ matrix $\Omega_{ij}=\delta_{ij}\frac{\eta_i}{\eta_{r-1+j}}$. Then
\[
V (\Xi_0 +\eta \eta^\top) V^\top=\diag(I_{r-1}, (\eta\eta^\top)_{r\leq i,j\leq p})
\]
and since $(\eta_k)_{r\leq k\leq p}\sim\mathcal N((\mu_k)_{r\leq k\leq p}, (\Sigma_{ij})_{r\leq i,j\leq p})$, it follows that $\eta \eta^\top$ has rank $1$ almost surely. Thus
$\rank(V(\Xi_0+\eta\eta^\top)V^\top)=r-1+1=r$ almost surely.

Now consider a random matrix $\Xi$. Clearly, $\rank(\Xi+\eta\eta^\top)\leq r$. The set $A_\Xi:=\{\rank(\Xi(\omega))=r-1\}$ is Borel,
since for $r=1$ it is precisely the set  $\{\tr(\Xi)=0\}$, and for $r>1$ one has $A_\Xi=\{e_{r-1}(\Xi)=0\}^c \cap \{ e_r(\Xi) =0\}$.
By assumption $\mathbb P[A_\Xi]>0$, thus the first part of the proof
implies
\[
\mathbb E[\rank(\Xi+\eta\eta^\top)\mid \rank(\Xi)=r-1]=r
\]
and thus $\rank(\Xi+\eta\eta^\top)=r$ almost surely on $A_\Xi$.
\end{proof}

\begin{lemma}\label{lem 1}
Suppose $\Xi_0\in \bar{\mathcal S}_p^+$ with $\rank(\Xi_0)= p-1$, and let $\Xi\sim\Gamma_p((p-1)/2,\Xi_0;\Sigma)$,
where $\Sigma$ is non-degenerate. Then
$\rank(\Xi)=p-1$ almost surely.
\end{lemma}
\begin{proof}
By Lemma \ref{mayx} \ref{x1}, the automorphism $\xi\rightarrow q^{-1}\xi q^{-1}$ with $q=\sqrt{\Sigma}$ yields $q^{-1}\Xi q^{-1}\sim\Gamma_p((p-1)/2,q^{-1}\Xi_0 q^{-1};I)$,
and since $\rank(\Xi_0)=\rank(q^{-1}\Xi_0 q^{-1})$, and $\Xi\simeq q^{-1}\Xi q^{-1}$, one may without loss of generality assume $\Sigma=2I$.

 Let $\mu_i \in\mathbb R^p$
for $i=1,\dots,p-1$ such that $\mu_1 \mu_1^\top+\dots \mu_{p-1}\mu_{p-1}^\top=\Xi_0$. Let $x_{ij}$, $1\leq i \leq p$, $1\leq j \leq p-1$
be a sequence of independent standard normally distributed random variables, and set $x_j=(x_{ij})_{1\leq i \leq p}$
and $y_j=x_j+\mu_j$. Then (see \cite[Section 1]{bib:mayer}) the random variable
\[
X=\sum_{j=1}^{p-1} y_j y_j^\top
\]
is $\Gamma_p(\frac{p-1}{2},\Xi_0; 2 I)$ distributed. Furthermore, $x:=(x_{ij})_{ij}$ has rank $p-1$ almost surely, hence $X$
has rank $p-1$ almost surely, and thus also $\Xi$.
\end{proof}
The following statement concerns the support of Wishart distributions with general shape parameter.
\begin{proposition}\label{prop may}
Suppose $\beta\in \{0,1/2,\dots,(p-2)/2\}$ and $\Sigma \in\mathcal S_p^+$. Suppose $\rank(\omega)= 2\beta+k$, where $1\leq k\leq p-(2\beta+1)$. Then $\Gamma_p(\beta,\omega;\Sigma)$, if exists, is supported in $D_{p}({2\beta})$. In other words, almost surely,
\begin{equation}\label{eq: support}
\rank(\Xi)\leq {2\beta}
\end{equation}
for any $\Xi\sim \Gamma_p(\beta,\omega;\Sigma)$.
\end{proposition}
\begin{proof}
Suppose first $\beta=0$ and $\rank(\omega)\geq 1$. Then, also $\Gamma_p(0,\widetilde\omega;2t I)$ exists, with $\rank (\widetilde\omega)=1$, see Lemma \ref{mayx} \ref{x3}. Let $x\in \bar{\mathcal S}_p^+$, then one can write
\[
x=\sum_{i=1}^p\mu_i\mu_i^\top,\quad \mu_i\in \mathbb R^p
\]
Let $t>0$ be fixed. By Lemma \ref{mayx} \ref{x3}, there exist independent random variables $\Xi_i\sim \Gamma_p(\beta=0,\mu_i\mu_i^\top;2tI)$, for $i=1,\dots,p$, and therefore
\[
\Xi=\Xi_1+\dots+\Xi_p\sim \Gamma_p(0,x;2tI),
\]
and thus a transition function of a Wishart semigroup with zero drift is constructed, violating the drift
condition for affine Markov processes on $\bar{\mathcal S}_p^+$ \cite[Theorem 2.4 and Definition 2.3, equation (2.4)]{bib:CFMT} (which rules out drifts strictly below $(p-1)/2$). Thus
$\Gamma_p(\beta,\omega;\Sigma)$ does not exist.

Let now $\beta\in \{1/2,\dots,(p-2)/2\}$, then, since $2\beta+k\geq 2\beta+1\geq 2$, there is nothing to show when $p\leq 2$. Set therefore $p\geq 3$. Then,
\begin{itemize}
\item $\beta':=(p-1)/2-\beta$ satisfies $1/2\leq \beta'\leq (p-2)/2$.
\item Since
\[
2\leq \rank(\omega)= 2\beta+k\leq 2\beta +(p-(2\beta+1))=p-1
\]
there exists $\omega'\in \bar{\mathcal S}_p^+$ with $\rank(\omega')=(p-1)-\rank(\omega)=(p-1)-(2\beta+k)$ and such that 
$\omega_*:=\omega+\omega'$ satisfies $\rank(\omega_*)= p-1$. Furthermore, since
\[
\rank(\omega')= p-1-(2\beta+k)=2\beta'-k\leq 2\beta'
\]
a random variable $Y\sim \Gamma_p(\beta',\omega';\Sigma)$ exists, independent of $\Xi$: Let
$m_i\in\mathbb R^p$ $(i=1,\dots,{n:=2\beta'})$ such that
\[
m_1m_1^\top+\dots+ m_n m_n^\top=\omega'
\]
and $\xi_j$ $(j=1,\dots,n$) be a sequence of independent, normally distributed random variables with mean $m_j$, and variance $\Sigma/2$, and independent of $\Xi$.
Then $Y:=\xi_1\xi_1^\top+\dots+\xi_n\xi_n^\top \sim \Gamma_p(\beta',\omega';\Sigma)$, see the remark following Definition \ref{def wish}.

\end{itemize}
The sum $\Xi'=\Xi+Y$ is $\Gamma_p((p-1)/2,\omega_*,\Sigma)$ distributed. Since $\rank(\omega_*)=p-1$, Lemma \ref{lem 1} applies and yields
 $\rank(\Xi')= p-1$ almost surely. Thus, by Lemma \ref{extra supp} (applied exactly $2\beta'$ times, since $Y$ is constructed by a sum of $2\beta'$ squares of independent, normally distributed vectors) one must have $\rank(\Xi)\leq 2\beta$ almost surely, as otherwise $\rank(\Xi')>p-1$ with non-zero probability.
\end{proof}
\subsection{Proof of the NCGS Conjecture.}\label{proof NCGS}
Proof of $\Leftarrow$:
{ 
Sufficiency of conditions in NCGS Conjecture was shown for 
$2\beta\in B$ in  \cite[Chap.38 (47), p.175]{john-kotz} and for $2\beta>p-1$
in \cite{bib:b91}. The case $2\beta=p-1$ follows from the case 
$2\beta>p-1$ by L\'evy continuity theorem arguments \cite{bib:mayer,bib:mayerJMA}.}

Proof of $\Rightarrow$:  Conversely, suppose the existence of a single distribution $\Gamma_p(\beta,\omega;I)$. Then by Lemma \ref{mayx} \ref{x3}, also 
$\Gamma_p(\beta,0;I)$ exists. Since the latter is a classical Wishart distribution with non-degenerate scale parameter, $\beta\in W_0$, the classical Gindikin set. \newline Let $\beta\in \{0,1/2,\dots,(p-2)/2\}$
and assume, for a contradiction, $\rank(\omega)=2\beta+l$, where $1\leq l\leq p-2\beta$. By Lemma \ref{mayx} \ref{x3} one can obtain non-central Wishart distributions for $\Gamma_p(\beta,\omega';\Sigma)$ with any $\rank(\omega')\leq 2\beta+l$ and any invertible $\Sigma$. 

Using, in addition, the support information of Proposition \ref{prop may}, one thus obtains a Wishart semigroup $(P_t)_{t\geq 0}$ with state space $D_{p}(2\beta+l)$ and with drift $2\beta$, by creating $\Gamma_p(\beta,x;2t I)$, for each $t>0$, and for each $x$ with $\rank(x)\leq 2\beta+l$.
Denote by $\mathcal A$ the infinitesimal generator of $(P_t)_{t\geq 0}$.

Distinguish the following two cases.
\begin{enumerate}
\item $l<p-2\beta$. Since for all $x\in D_p(2\beta+l)$, $e_{2\beta+l+1}{(x)=} 0$, 
\begin{align*}
0&=\lim_{t\rightarrow 0}\frac{P_t e_{2\beta+l+1}(x)-e_{2\beta+l+1}(x)}{t}=\mathcal A e_{2\beta+l+1}(x)=\\
&= (p-(2\beta+l))(-\beta-l)e_{2\beta+l}(x)\neq 0,\quad \text{for all } x \text{ with }\rank(x)=2\beta+l,
\end{align*}
which is a contradiction. Here, for the last identity Proposition \ref{prop GM1} has been used.
\item  $l=p-2\beta$. Then $\rank(\omega)=p$ and the semigroup $(P_t)_{t\geq 0}$ acts on $C_0(\bar{\mathcal S}_p^+)$. The positivity of the Feller semigroup implies that its infinitesimal generator $\mathcal A$ satisfies the positive maximum principle. Applied to $e_p(x)=\det(x)$ this implies
that
\[
\mathcal A \det(x_0)\geq 0
\]
for any $x_0$ with $\rank(x_0)<p$.  Choose $x_0$ with $\rank(x_0)=p-1$, then
$e_{p-1}(x_0)>0$, and therefore by Proposition \ref{prop GM1} (setting $n=p$ and recalling {that} $e_p=\det$)
\[
\mathcal A \det(x_0)=(2\beta-p+1)e_{p-1}(x_0)<0
\]
because $\beta\in\{0,1\dots,\frac{p-2}{2}\}$, by assumption. This violates the positive maximum principle. 
\end{enumerate}

These two contradictions imply that indeed $\rank(\omega)\leq 2\beta$, whenever $\beta\in\{0,\dots,\frac{p-2}{2}\}$, and thus the proof of the NCGS conjecture is finished.
{
\begin{remark}
Let us mention another proof of the necessity in  the NCGS. 
As above, the  existence of a single distribution $\Gamma_p(\beta,\omega;I)$ implies the  existence of
 a Wishart semigroup $(P_t)_{t\geq 0}$ with state space $D_{p}(2\beta+l)$ and with drift $2\beta$. By Proposition A.2(ii),   the Wishart SDE \eqref{eq:Wishart:SDePLUS} has a global weak solution
 with $X_0=\omega$.  The proof is completed by using Theorem \ref{char sdes}.
	\end{remark}
}
\subsection{A Characterization of Wishart Semigroups}\label{char low rank}
The paper concludes with the following characterization of Wishart semigroups with state spaces {$D_p(k)$, the $p\times p$ symmetric
positive semi-definite matrices of rank $\leq k$.\footnote{Note that $D_p(k)$ are non-convex domains for $k<p$, but, {by Theorem \ref{th super}, }
Wishart semigroups on $D_p(k)$ cannot be extended to their convex hull $\overline{\mathcal S}_p^+$.} The statement has been conjectured by Damir Filipovi\'c \cite{bib:filcon} in 2009.
\begin{theorem}\label{nude semigroups}
Let $k\in\{1,\dots,p\}$ and let $\alpha \geq 0$.
The following are equivalent:
\begin{enumerate}
\item \label{state2} The Wishart semigroup with state-space $D=D_p(k)$ exists.
\item \label{state1} If $k\in\{1,\dots,p-1\}$, then $\alpha= k$, and if $k=p$, then $\alpha\geq p-1$.
\end{enumerate}
\end{theorem}
\begin{proof}
If $k=p$, that is $D=\bar{\mathcal S}_p^+$, then $\alpha\geq p-1$ due to
\cite{bib:CFMT}, which also includes a proof of existence. Therefore, only the cases $k<p$ require a proof:

Proof of \ref{state1} $\Rightarrow$ \ref{state2}: The existence is shown by construction, using squares. See, for instance, the proof of Theorem \ref{char sdes}, or \cite[Examples {III.1 and III.2}]{bib:mayer}.

Proof of \ref{state2} $\Rightarrow$ \ref{state1}: Assume the existence of a Wishart semigroup on $D_p(k)$
\footnote{Using the NCGS conjecture, the following, weaker, conclusion can be made. Assume the existence of a Wishart semigroup on $D_p(k)$. Then $\Gamma_p(\alpha,x_0,I)$ exists with $\rank(x_0)=k$. By the NCGS Conjecture, $\alpha/2 \in W_0$ and, if $\alpha<p-1$, then $\rank(x_0)\leq \alpha$. This implies $\alpha \geq k$.}. Since $e_{k+1}$ vanishes on $D_p(k)$, one obtains by using Proposition \ref{prop GM1} that
\[
0=(\mathcal A e_{k+1})(x)=(p-k)(\alpha-k) e_k(x).
\]
Since $k<p$, and $e_k(x)>0$ for $\rank(x)=k$ matrices, $\alpha$ must be equal to $k$. 
\end{proof}
\begin{appendix}

\section{Wishart Semimartingales}\label{appendix a}

\begin{proposition}\label{prop wish semi}
Let $(\Omega,\mathcal F,(\mathcal F_t)_{t\geq 0}\mathbb P)$ be a {standard} filtered probability space. Let $(X_t)_{t\geq 0}$ be a continuous, $\bar{\mathcal S}_p^+$ valued semimartingale of the form
\begin{equation}\label{eq: wish semi}
dX_t=dM_t+\alpha { I dt} , 
\end{equation}
where $\alpha\geq 0$, and the continuous  martingale $M_t$ has quadratic variation
\begin{equation}\label{eq: qvar}
d{\langle}M_{t,ij}, M_{t,kl}{\rangle}
=\left((X_t)_{ik}\delta_{jl}+(X_t)_{il}\delta_{jk}+(X_t)_{jk}\delta_{il}+(X_t)_{jl}\delta_{ik}\right)dt.
\end{equation}
Then there exists an extension $(\widetilde{\Omega},\widetilde{\mathcal F},(\widetilde{\mathcal F_t})_{t\geq 0}, \widetilde{\mathbb P})$ of $(\Omega,\mathcal F,(\mathcal F_t)_{t\geq 0}, \mathbb P)$ which supports
a $d\times d$ standard Brownian motion $W$ such that
\begin{equation}\label{eq: matrix sde}
dX_t=\sqrt{X_t}dW_t+dW_t^\top \sqrt{X_t}+\alpha I dt.
\end{equation}
\end{proposition}
\begin{proof}
This is an application of \cite[Theorem V.20.1]{rogerswilliams2}, where one interprets the SDE
\eqref{eq: matrix sde} in vector form, and thus $W$ as a vector of $p^2$ independent, standard Brownian motions. The details of the proof are the same as those
found in  \cite[p. 53, Proof of Theorem 2.6]{bib:CFMT}.
\end{proof}

\begin{proposition}\label{prop a22}
Let $D\subset\mathcal S_p^+$ such that $D_p(1)\subset D$, and let $(P_t)_{t\geq 0}$ be a Wishart semigroup on $D$ with parameter $\alpha$. The following hold:
\begin{enumerate}
\item \label{crux 1} 
For $x\in D$ let $(X,\mathbb P^x)$ be the canonical representation of the Markov semigroup with the initial law ${\delta_x}$ (cf. Remark \ref{rem1x}(i)). There exists a version $\widetilde X$ of $X$ that is a continuous semimartingale of the form \eqref{eq: wish semi} with quadratic variation \eqref{eq: qvar}. 
\item \label{crux 2} For any $x\in D$, the Wishart SDE \eqref{eq:Wishart:SDePLUS} has a global weak solution
with $X_0=x$.
\end{enumerate}
\end{proposition}

\begin{proof}
Proof of \ref{crux 1}: The canonical representation $(X, (\mathbb P_x)_{x\in D}$ constitutes a time homogeneous Markov process in the sense of \cite[Definition 1]{Cpaths} and an affine processes
in the sense of \cite[Definition 2]{Cpaths}. Since $D_p(1)\subset D$, $D$ contains {$p\times (p+1)/2+1$} affinely independent elements, and thus $D$ satisfies \cite[Assumption 1]{Cpaths}.\\
Let $\mathcal F_t^0=\sigma(X_s,s\leq t)$ be the filtration generated by the canonical process $X_t(\omega):=\omega(t)$, and let $\mathcal F^0:=\vee _{t\geq 0}\mathcal F_t$. Then by \cite[Theorem 2]{Cpaths},
there exists a version $\widetilde X$ of $X$ which is c\`adl\`ag. Since $(P_t)_{t\geq 0}$ is conservative,
\cite[Theorem  6]{Cpaths} implies that $\widetilde X$ is a semimartingale with characteristics $(B,C,\nu)$,
where
\begin{align*}
B_{t,i}&=\int_0^t b_i(\widetilde X_{s_-})ds,\\
C_{t,ij}&=\int_0^tc_{ij}(\widetilde X_{s-})ds,\\
\nu(\omega;dt,d\xi)&=K(\widetilde X_t,d\xi) dt.
\end{align*}
Here $b: D\rightarrow \mathcal S_p$ and $c: D\rightarrow \text{Sym}_+(\mathcal S_p)$ are measurable functions, and $K(x,d\xi)$ is a positive kernel ($\text{Sym}_+(V)$ denotes positive semidefinite matrices
on a vector space $V$).
From the computations in the proof of Proposition \ref{prop genx} it follows that $(X,\mathbb P_x)$ is regular in the sense of \cite[Definition 7]{Cpaths}, that is, the coefficients $\phi,\psi$ are differentiable
at $t=0$, with derivatives $F(u), R(u)$ given by \eqref{eq F} and \eqref{eq R}. On the other hand, by \cite[Theorem 7]{Cpaths}, the functions $F(u), R(u)$ uniquely determine the differential characteristics $b_i(x), c_{ij}(x)$ and $K(x,d\xi)$. A comparison of \eqref{eq F}--\eqref{eq R} with the expressions of $F$ and $R$ in \cite[Theorem  7]{Cpaths} finally reveals that $\nu=0$, i.e., the process $\widetilde X$ is continuous $\mathbb P_x$-almost surely, because by the semimartingale decomposition
\[
X_t=X_0+\int_0^t dM_s+\int _0^tb(X_s)ds,
\]
where $M$ is the continuous martingale part of $X$.

Proof of \ref{crux 2}:   Follows from \ref{crux 1} by applying Proposition \ref{prop wish semi}.

\end{proof}
\section{Fourier-Laplace Transform of  Wishart distributions}
This section shows that the Laplace transform \eqref{FLT Mayerhofer Wishart} can be extended to its maximal domain, which is dictated by the blow up of the right side.

The right side of  \eqref{FLT Mayerhofer Wishart} is a real analytic function, which is finite on the domain
\[
D(\mu):=-\Sigma^{-1}+\mathcal S_p^+
\]
but blows up as the argument $u$ approaches the boundary $\partial D(\mu)$, since then the determinant vanishes.

Furthermore, the right side of \eqref{FLT Mayerhofer Wishart} can be extended to a complex analytic function on the complex strip $D(\mu)+i\mathcal S_p$ (by just replacing $u$ by $u+iv$, where $v\in\mathcal S_p$) and it agrees, by definition, with the left side of \eqref{FLT Mayerhofer Wishart}, on a set of uniqueness, namely the open domain $\mathcal S_p^+$. Hence, by \cite[(9.4.4)]{dieudonne}, equality holds in \eqref{FLT Mayerhofer Wishart} for $u\in \mathcal S_p^++i\mathcal S_p$.

The following extends the validity of \eqref{FLT Mayerhofer Wishart} to
its maximal domain $D(\mu)+i\mathcal S_p$:

\begin{proposition}\label{FLT maximal}
Let $\mu=\Gamma_p(\beta,\omega;\Sigma)$. Then its Fourier-Laplace transform
can be extended to the complex strip $D(\mu)+i\mathcal S_p$, and
\eqref{FLT Mayerhofer Wishart} holds for any $u\in D(\mu)+i\mathcal S_p$.
\end{proposition}

For the proof, the following fundamental technical statement concerning extension of the Laplace transform of a measure on the non-negative real line is used. It is a refinement of \cite[Lemma A.4]{ADPTA}:
\begin{lemma}\label{lem adpta}
Let $\mu$ be a probability measure on $\mathbb R_+$, and $h$ an analytic function on $(-\infty,s_1)$, where $s_1>s_0\geq 0$ such that
\begin{equation}\label{eq A1}
\int _{\mathbb R_+}e^{sx} \mu(dx)=h(s)
\end{equation}
for $s\in (-\infty, s_0)$. Then \eqref{eq A1} also holds for $s\in (-\infty, s_1)$. 
\end{lemma}
\begin{proof}
If $s_0>0$, the statement follows from \cite[Lemma A.4]{ADPTA}. Let therefore $s_0=0$.

Denote, {for $s\le 0$}, $f(s)=\int_{\mathbb R_+} e^{sx}\mu(dx)$.

 Since $h(s)$ is real analytic at $0$, there exists $0<\varepsilon<s_1$ such that for any $s\in (-\varepsilon,\varepsilon)$ 
\[
h(s)=\sum_{k\geq 0} \frac{c_k}{k!} s^k.
\]
Furthermore, by dominated convergence, one obtains iteratively for the left derivatives
\[
\int_{\mathbb R_+} x^k e^{sx} \mu(dx)= \lim_{t\uparrow 0} \int _{\mathbb R_+}x^{k-1}e^{sx} \frac{e^{-tx}-1}{-t}\mu(dx)=f^{(k)}(s)=h^{(k)}(s), \quad s\leq 0,
\]
hence
\[
c_k=\int_{\mathbb R_+}x^k \mu(dx).
\]
Hence, by monotone convergence, for any $s\in (0,\varepsilon)$
\[
h(s)=\sum_{k\geq 0} \int_{\mathbb R_+}\frac{s^k {x^k}}{k!} \mu(dx)=\int_{\mathbb R_+}\sum_{k\geq 0}\frac{s^k {x^k}}{k!} \mu(dx)=\int_{\mathbb R_+} e^{sx}\mu(dx).
\]
Thus $h(s)$ {verifies \eqref{eq A1}}  on all of $(-\infty, \varepsilon)$. Now the assumptions of
\cite[Lemma A.4]{ADPTA} are verified (setting $s_0=\varepsilon$), that shows the extension to the maximal domain $(-\infty, s_1)$.
\end{proof}

{\it Proof of Proposition \ref{FLT maximal}.}
For $u=\Sigma^{-1}$, define $\mu^*$ {as} the pushforward of $\mu=\Gamma_p(\beta,\omega;\Sigma)$ under
$\xi\mapsto \tr(u \xi)=\tr(\Sigma^{-1}\xi)$. Then $\mu^*$ is a probability measure on $\mathbb R_+$ with
 Laplace transform
\begin{align}\label{eq: one dim}
f(t):&=\int e^{t x}\mu^*(dx)=\int e^{-\tr{((-tu)\xi)}}\mu(d\xi)\\\nonumber&=(\det\Sigma)^{-\beta} \det(\Sigma^{-1}(1-t))^{-\beta}e^{{t(1-t)^{-1}\tr
	(\Sigma^{-1}\omega)}},\quad\quad\quad t\leq 0,
\end{align}
and the right side is real analytic for $t<1$. Hence, by Lemma \ref{lem adpta} the left side is also
finite for $t<1$ and equality holds in \eqref{eq: one dim}.

Therefore, it is shown that the formula \eqref{FLT Mayerhofer Wishart} can be extended to $u=-t\Sigma^{-1}$, for any
$t<1$. Since $u>-\Sigma^{-1}$ implies $u>-t\Sigma^{-1}$ for some $t<1$, also for any $u>-{\Sigma^{-1}}$
\[
\int e^{-\tr(u\xi)}\mu(d\xi)\leq \int e^{t \tr(\Sigma^{-1}\xi)}\mu(d\xi)<\infty
\]
and therefore the left side of \eqref{FLT Mayerhofer Wishart} exists for any $u>-\Sigma^{-1}$, and thus also the Fourier-Laplace transform exists for any $u+iv$, where $u>-\Sigma^{-1}$ and $v\in\mathcal S_p$. Since the Fourier-Laplace transform is complex analytic on the strip ${-{\Sigma}^{-1}}+ \mathcal S_p^++i\mathcal S_p$, and agrees with the right side
of \eqref{FLT Mayerhofer Wishart} on the domain $\mathcal S_p^+$ (which is a set of uniqueness),
equality in \eqref{FLT Mayerhofer Wishart} holds by \cite[(9.4.4)]{dieudonne}. This concludes
the proof of Proposition \ref{FLT maximal}.

\end{appendix}

\end{document}